\documentclass[12pt]{amsart}
\usepackage{amsmath,amsthm,amsfonts,amssymb,verbatim,eucal}
\usepackage[all]{xy}

\numberwithin{equation}{section}

\newtheorem{thm}[equation]{Theorem} 
\newtheorem{prop}[equation]{Proposition}
\newtheorem{lemma}[equation]{Lemma} 
\newtheorem{cor}[equation]{Corollary}
\newtheorem{example}[equation]{Example}
\newtheorem{remark}[equation]{Remark}

\DeclareMathOperator{\Alt}{Alt}
\DeclareMathOperator{\Ext}{Ext}
\DeclareMathOperator{\gr}{gr} 
\DeclareMathOperator{\Ima}{Im}
\DeclareMathOperator{\im}{Im}

\DeclareMathOperator{\Span}{Span}
\DeclareMathOperator{\Tor}{Tor}

\newcommand{\NN}{\mathbb N}

\newcommand{\DOT}{\setlength{\unitlength}{1pt}\begin{picture}(2.5,2)
               (1,1)\put(2,3.5){\circle*{3}}\end{picture}}

\newcommand{\coh}{{\rm H}}

\newcommand{\id}{\mbox{\rm id\,}}
\newcommand{\Hom}{\mbox{\rm Hom\,}}

\renewcommand{\ker}{\mbox{\rm Ker\,}}

\newcommand{\ot}{\otimes}

\newcommand{\cH}{\mathcal{H}}
  
\newcommand{\ZZ}{\mathbb{Z}}

\newcommand{\tensor}{\otimes}

\newcommand{\HH}{{\rm HH}}
\newcommand{\Wedge}{\textstyle\bigwedge}

\newcommand{\LH}{\text{LH}}

\begin{document}
\begin{abstract}
Braverman and Gaitsgory
gave necessary and sufficient conditions
for a nonhomogeneous
quadratic algebra to satisfy the Poincar\'e-Birkhoff-Witt property
when its homogeneous version is Koszul.
We widen their viewpoint and consider a 
quotient of an algebra that is free over some 
(not necessarily semisimple) subalgebra.
We show that their theorem holds under
a weaker hypothesis: We require the homogeneous version 
of the nonhomogeneous quadratic algebra to be
the skew group algebra (semidirect product algebra)
of a finite group acting on a Koszul algebra,
obtaining conditions for the Poincar\'e-Birkhoff-Witt property 
over (nonsemisimple) group algebras.
We prove our main results by exploiting a double complex
adapted from Guccione, Guccione, and Valqui
(formed from a Koszul complex and a resolution of the group), 
giving a practical way to analyze Hochschild cohomology
and deformations of skew group algebras
in positive characteristic.
We apply these conditions to graded Hecke algebras and 
Drinfeld orbifold algebras 
(including rational Cherednik algebras and symplectic
reflection algebras) in arbitrary characteristic, 
with special interest in 
the case when the characteristic of the underlying field
divides the order of the acting group.
\end{abstract}
\title[Poincar{\' e}-Birkhoff-Witt Theorem] 
{A Poincare-Birkhoff-Witt Theorem \\
for Quadratic Algebras with Group Actions}

\date{September 25, 2012}
\author{A.\ V.\ Shepler}
\address{Department of Mathematics, University of North Texas,
Denton, Texas 76203, USA}
\email{ashepler@unt.edu}
\author{S.\  Witherspoon}
\address{Department of Mathematics\\Texas A\&M University\\
College Station, Texas 77843, USA}\email{sjw@math.tamu.edu}
\thanks{The first author was partially supported by NSF grants
\#DMS-0800951 and \#DMS-1101177.
The second author was partially supported by
NSF grants 
\#DMS-0800832 and \#DMS-1101399.}

\maketitle

\section{Introduction}

Poincar\'e-Birkhoff-Witt properties are used
to isolate convenient canonical bases of algebras,
identify algebras with formal deformations,
and depict associated graded structures of algebras explicitly.
They reveal how a set of generators for an ideal of relations 
defining an algebra may capture the graded structure implied
by the entire ideal. In 1996, Braverman and Gaitsgory~\cite{BG}  
gave necessary and sufficient conditions for a nonhomogeneous 
quadratic algebra to satisfy the Poincar\'e-Birkhoff-Witt
property when a homogeneous version is Koszul.
Poincar\'e-Birkhoff-Witt theorems are often established by inspection of monomial
orderings, noncommutative Gr\"obner bases, or other methods from
noncommutative computational algebra.  In contrast, Braverman and Gaitsgory used
Hochschild cohomology to streamline arguments.
They identified conditions that often
provide an elegant alternative to direct application
of Bergman's Diamond Lemma~\cite{Bergman} (which can be tedious)
for determining when a Poincar\'e-Birkhoff-Witt property holds.
Thus their theory has been used widely to investigate various
algebras in many settings.
(See~\cite{PolishchukPositselski} and references therein.)

Etingof and Ginzburg~\cite{EtingofGinzburg}
noted that the conditions
for a Poincar\'e-Birkhoff-Witt property developed by 
Braverman and Gaitsgory~\cite{BG} may be generalized
by replacing the ground field $k$ by any semisimple ring
(for example, a group ring $kG$) using the theory of 
Beilinson, Ginzburg, and Soergel~\cite{BGS} of Koszul rings.
They exploited this generalization 
in investigations of symplectic reflection algebras.
Later others applied it even more generally, for example,
Halbout, Oudom, and Tang~\cite{HOT}.
Yet the theorem of Braverman and Gaitsgory
and the generalization observed by 
Etingof and Ginzburg~\cite{EtingofGinzburg}
do not apply in some interesting settings,
in particular, when replacing the ground field
by a ring that is not semisimple.
For example, a theory of symplectic reflection algebras and graded Hecke algebras
over fields of arbitrary characteristic would include
exploration of quotient algebras over group rings $kG$ 
when the characteristic of the field $k$ divides the
order of the finite group $G$.

We show in Theorem~\ref{thm:main} that the conditions of Braverman and Gaitsgory
for a Poincar\'e-Birkhoff-Witt property hold over 
{\em any} finite group algebra
(including nonsemisimple group algebras).
We establish necessary and 
sufficient conditions for a filtered
quadratic algebra over $kG$ to satisfy the PBW property when its
homogeneous version is a skew group algebra (semidirect
product algebra) formed from a finite group
$G$ acting on a Koszul algebra $S$. 
The proof uses deformation theory and 
an explicit bimodule resolution $X_{\DOT}$.
This resolution, defined in Section~\ref{sec:resolution}
and adapted from Guccione, Guccione, and Valqui~\cite{GGV}, is 
comprised of both the Koszul resolution for $S$ and the bar resolution
for $kG$. 
Semisimplicity does not play a crucial role. 
However, there are potentially many graded deformations of the 
skew group algebra $S\# G$ in positive characteristic that this theory 
does not identify,
corresponding to components of the resolution $X_{\DOT}$ 
that are not fully explored here. 
Other methods have been employed in the modular setting in some
special cases; for example, 
see \cite{Griffeth} for representation-theoretic techniques
and \cite{doa} for application of the Diamond Lemma. 

We compare our techniques with those in the nonmodular setting: 
In characteristic zero, the group algebra $kG$ has trivial cohomology and thus
the Koszul resolution of $S$ is sufficient for analyzing the 
Hochschild cohomology of  $S\# G$
(by a spectral sequence argument, see \cite{Farinati}).
But in positive characteristic, the group algebra may exhibit nontrivial cohomology
influencing the Hochschild cohomology of $S\# G$, and one seeks a more sophisticated
replacement for the Koszul resolution of $S$ which nevertheless remains practical
for determining concrete results (such as 
establishing a Poincar\'e-Birkhoff-Witt property).  
In this article, we highlight the complex $X_{\DOT}$
as a first tool in understanding the Hochschild cohomology
of $S\# G$ in arbitrary characteristic and its deformation theory. 
Our general construction of $X_{\DOT}$ in Section \ref{sec:resolution} takes
as input arbitrary resolutions of $S$ and of $kG$ satisfying some 
natural conditions;
the Koszul resolution of $S$ (in case $S$ is Koszul) and the bar resolution
of $kG$ are but two of the potentially many useful choices. 
Our approach may lead to better understanding of many algebras
currently of interest, such as rational Cherednik algebras in positive
characteristic.
In Section~\ref{DOAs} we illustrate these applications by focusing
on a collection of such algebras, showing how known results in
the modular setting may now
be obtained directly from the methods of Braverman and Gaitsgory.

There are several  papers containing generalizations of Koszul
algebras for various purposes, those most relevant to our setting
being~\cite{GRS,Liping,Woodcock}. 
We will not need these here, however, as the known properties of Koszul algebras
over fields are sufficient to obtain our results. 
It would be interesting to determine whether it is possible to generalize
the theory of Braverman and Gaitsgory more directly by 
using some equivalent definition of Koszul rings over
(nonsemisimple) rings involving a bimodule complex (cf.\ 
\cite[Proposition A.2(b)]{BG}).

Throughout this article, $k$ denotes a field (of arbitrary 
characteristic)
and tensor symbols without subscript denote tensor product over $k$: 
$\otimes = \otimes_k$.  (Tensor products over other rings will always
be indicated). We assume all $k$-algebras have unity
and all modules are left modules, unless otherwise specified.
We use the notation $\NN=\ZZ_{\geq0}$, the
nonnegative integers.

\section{Filtered and homogeneous quadratic algebras}

Quadratic algebras and their variations
traditionally arise from taking a free algebra
modulo a set of (nonhomogeneous) relations of
degree two.
Examples include symmetric algebras, commutative polynomial rings,
skew/quantum polynomial rings, Weyl algebras,
Clifford and exterior algebras, 
and enveloping algebras of finite dimensional Lie algebras.
But quadratic shape arises 
from any $\NN$-graded algebra modulo
an ideal generated in filtered degree two.
Specifically,
let $T=\bigoplus_m T^m$ be an $\NN$-graded $k$-algebra.
(We have in mind the tensor algebra (i.e., free algebra) 
$T_k(V)$ of a $k$-vector space $V$,
or the tensor algebra $T_B(U)$ of a $B$-bimodule $U$ 
over a $k$-algebra $B$, for
example, a group algebra.)
Let $F^m(T) = T^0\oplus T^1 \oplus \ldots \oplus T^m$ 
be the $m$-th filtered component
of $T$ and
fix a two-sided ideal $I$ of $T$.

A quotient $T/I$ is called a 
{\bf filtered quadratic algebra}
(or a {\em nonhomogeneous quadratic algebra})
if $I$ can be generated in filtered degree 2, i.e., if
$I=\langle P \rangle$ for some $P\subset F^2(T)$.
In this case, $P$ is called a set of
{\bf filtered quadratic relations}.
Quadratic algebras are filtered algebras, with $m$-th filtered component
$F^m(T/I) = (F^m(T)+I)/I$ induced from that on $T$.

When a set of filtered quadratic relations $R$ resides exactly in
the second graded component (homogeneous of degree 2),
i.e., $R\subset T^2$, 
we call the quotient 
$T/\langle R \rangle$ 
a {\bf homogeneous quadratic algebra}.  
Homogeneous quadratic algebras are not only filtered, but
also graded: 
$T/I=\oplus_m (T^m+I)/I$
for $I=\langle R \rangle$.

One may easily associate to every filtered quadratic
algebra $T/I$ two different graded versions.
We might cross out 
lower order terms in a {\em generating}
set of relations for the algebra, or,
instead, cross out lower order terms 
in each element of the {\em entire} ideal of relations.
The Poincar\'e-Birkhoff-Witt conditions are precisely those 
under which these two 
graded versions of the original algebra 
coincide, as we now recall.

On one hand, we may 
simply ignore those parts of each defining relation in $P$
of degree less than two to obtain a simplified version of the original 
filtered algebra
which is homogeneous quadratic.
Formally, we let $\pi$ denote projection onto the
second graded component of the tensor algebra:
$
   \pi:\ T \rightarrow T^2.
$
If $P$ is a set of filtered quadratic relations defining
the algebra $T/I$,
then $\pi(P)$ defines a graded quadratic algebra
$$T/\langle \pi(P) \rangle\ ,$$ called the
{\bf homogeneous quadratic algebra
determined by $P$} (or the {\em induced
homogeneous quadratic algebra}).
Although this construction depends on the choice of
generators $P$ for $I$,
often a choice can be made 
to capture the degree two aspect of the entire ideal
of relations.  

On the other hand, we may consider the traditional {\bf associated graded
algebra} 
$$\gr\big( T/I\big) 
= \ \bigoplus_{m} \mathcal{F}_m/ \mathcal{F}_{m-1} \, ,
$$
where $\mathcal{F}_m=F^m(T/I)$.
This graded version of the original
algebra does not depend on the choice of generators $P$ of the ideal
$I$ of relations.
We realize the associated graded algebra concretely also as
a homogeneous version of the original filtered quadratic algebra by
projecting each element in the ideal $I$ onto its
leading homogeneous part and taking
the quotient by the resulting ideal:
The
{\bf associated homogeneous quadratic algebra}
(also called the {\em leading homogeneous algebra})
is defined as
$$T/\langle \LH(I)\rangle\, , $$ 
where
$\LH(I)=\{\LH(f): f \in I\}$ and $\LH(f)$ picks off the highest
homogeneous part of $f$ in $T$. (For $f=\sum_{i=1}^d f_i$ with 
each $f_i$ in $T^i$ and $f_d$ nonzero, $\LH(f)=f_d$, and $\LH(0)=0$.)
This associated homogeneous quadratic algebra 
is isomorphic as a graded
$k$-algebra to the associated graded algebra
(see Li~\cite[Theorem 3.2]{Li2012}):
$$
\gr \big(T/I\big)
\ \cong\
T/ \langle \LH(I)\rangle\, .
$$

We say the original filtered quadratic algebra
has {\em Poincar\'e-Birkhoff-Witt type} 
when the associated homogeneous quadratic algebra
 and the homogeneous quadratic algebra determined by $P$
coincide, and thus both give the associated graded algebra.
(This terminology
arises in analogy with the original PBW Theorem
for universal enveloping algebras of Lie algebras.)  More precisely, 
let 
$$\phi: T \rightarrow \gr(T/I)$$
be the natural $k$-algebra epimorphism.
Then (again, see Li~\cite[Theorem 3.2]{Li2012})
$$\ker \phi = \langle \LH(I) \rangle\, , $$
which contains $\langle \pi(P) \rangle$
when $I=\langle P \rangle$.
Thus, a natural surjection $p$ always arises from the
homogeneous quadratic algebra determined by $P$
to the associated graded algebra of the filtered quadratic algebra:
$$p:\ T /\langle \pi(P) \rangle 
\rightarrow 
\gr \big(T /\langle P \rangle\big)\ .$$ 
A filtered quadratic algebra exhibits {\bf PBW type}
(with respect to $P$)
exactly when it can be written as $T /\langle P\rangle$
for some set of filtered quadratic relations $P$
with $p$ an isomorphism of graded algebras,
i.e., when
$$\langle \pi(P)\rangle = \ker \phi = \langle \LH(I)\rangle
$$
and thus the homogeneous versions of the filtered quadratic algebra
all coincide: 
$$
T /\langle \pi(P) \rangle 
\cong
\gr \big(T /\langle P \rangle\big)\ 
\cong
T /\langle \LH \langle P\rangle \rangle \,
.$$
In this case, we say that $P$ is a set of {\bf PBW generating relations}.

The definition of PBW type depends on $P$, as we see in Example~\ref{cuteexample} 
below. 
But if $T^0=k$, 
we may require (without loss of generality) that $P$ be a
$k$-subspace, in which case a set of PBW generating relations
$P$ is unique.  
In fact, a set of PBW generating relations is always 
unique up to additive closure
over the degree zero component of $T$, as we explain in the next
proposition:

\begin{prop}\label{uniquePBWrelations}
PBW filtered quadratic algebras have unique PBW 
filtered quadratic relations up to addition and multiplication by
degree zero elements:
If $P,P'$ are each PBW filtered quadratic relations defining the same
filtered quadratic algebra $T/I$ (that is,
$\langle P \rangle = I = \langle P' \rangle$),
then $P$ and $P'$ generate the same $T^0$-bimodule.
Thus, if both $P$ and $P'$ are closed under addition 
and under multiplication by degree zero elements in $T$, then $P=P'$.
\end{prop}

\begin{proof}
We check that for any set $P$ of PBW relations, 
\begin{equation}\label{rawIJ}
\langle P\rangle\cap F^1(T)= \{0\} \ \ \ \text{and} \ \ \ 
\langle P\rangle\cap F^2(T)= T^0 P T^0,
\end{equation}
where $T^0PT^0$ is the $T^0$-bimodule generated by $P$.
The first claim is immediate:
Since $\pi(P)\subset T^2$, the algebra $T/\langle \pi(P) \rangle$
contains $F^1(T)$ as a subspace;
if $\langle P\rangle \cap F^1(T)\neq \{0\}$, 
then $\gr(T/\langle P \rangle)$ does not contain
$F^1(T)$ as a subspace, and so $p$ cannot be an isomorphism.

Now let $x$ be any element of $\langle P \rangle \cap F^2(T)$
with $x\not\in T^0PT^0$. 
Note that for all $y\in T^0PT^0$, $\pi(x)\neq \pi(y)$.
Otherwise, some nonzero $x-y$ would lie in $F^1(T)$ (as $\pi(x-y)=0$), 
implying that $\langle P \rangle \cap F^1(T)$
is nonzero.
Hence $\pi(x)\not\in \pi(T^0PT^0)=T^0\pi(P)T^0$.  
But this contradicts the PBW property, which implies that
$\pi(x)\in\langle\text{LH}\langle P \rangle \rangle=\langle \pi(P) \rangle$
(a homogeneous ideal) and hence $\pi(x)\in T^0\pi(P)T^0$.

Lastly, if $P$ and $P'$ are both PBW filtered quadratic relations
generating the ideal $I$, then
$
T^0PT^0=I\cap F^2(T)=T^0P'T^0$,
and $P$ and $P'$ generate the same $T^0$-bimodule.
\end{proof}

\begin{example}\label{cuteexample}{\em 
We give an example of a filtered quadratic algebra that exhibits PBW type
with respect to one generating set of relations but not another.
Let $V=\Span_k\{x,y\}$ and let $I$ be the two-sided ideal
in the tensor algebra $T=T(V)$ (over $k$) 
generated by $P=\Span_k\{xy-x,yx-y\}$ (we suppress
tensor signs in $T$).
The filtered quadratic algebra $T(V)/I = T(V)/\langle P \rangle$
is not of PBW type with respect to $P$.  Indeed, the
graded summand of the associated graded algebra of degree 2 has dimension
0 over $k$ since a quick calculation verifies that
$x^2-x$ and $y^2-y$ both lie in the ideal $\langle P \rangle$,
implying that all of $T^2(V)=V\ot V$ 
projects to zero in $T(V)/\langle P \rangle$
after passing to the associated graded algebra.
But $\pi(P)=\Span_k\{xy,yx\}\subset V\otimes V$
defines a homogeneous quadratic algebra $T(V)/\langle xy,yx \rangle$
whose graded summand of degree 2 has dimension 2 over $k$
(with basis $\{x^2, y^2\}$).
Note that the filtered quadratic algebra
$T(V)/I$ exhibits PBW type with respect to a different set of generators
for the ideal $I$.  If we extend $P$ to 
$$P'=\Span_k
\{xy-x, yx-y,x^2-x, y^2-y  \}\, ,$$ then
$T(V)/I$ is of PBW type with respect to $P'$,
as its associated graded algebra is isomorphic
to
$$T(V)/\langle xy,yx,x^2,y^2\rangle \cong T(V)/\langle \pi(P')\rangle \ .$$
}\end{example}

\quad

In Section~\ref{sec:main} we will analyze 
filtered and homogeneous quadratic algebras
when $T$ is a free algebra.  Traditional quadratic algebras
arise as quotients of a 
tensor algebra $T$ of a $k$-vector space $V$.  
We expand this view and consider (more generally) bimodules 
$V$ over an arbitrary $k$-algebra $B$
in order to include constructions of algebras of quadratic shape 
naturally appearing in other settings.  
First we recall 
Hochschild cohomology and deformations in the next section,
and construct the needed resolution $X_{\DOT}$ in Section~\ref{sec:resolution}.

\section{Deformations and  Koszul algebras}\label{sec:relative}

Let $A$ be a $k$-algebra. 
Let $M$ be an $A$-bimodule, equivalently, 
an $A^e$-module, where $A^e=A\ot A^{op}$,
the enveloping algebra of $A$.
As $k$ is a field, the   Hochschild cohomology of $M$ is
$$
  \HH^n (A,M) = \Ext^n_{A^e}(A,M).
$$
This cohomology can be examined explicitly using
the  bar resolution, a free resolution
of the $A^e$-module $A$: 
\begin{equation}\label{relative-bar}
 \cdots \stackrel{\delta_3}{\longrightarrow} A\ot A\ot A\ot A
  \stackrel{\delta_2}{\longrightarrow} A\ot A\ot A
  \stackrel{\delta_1}{\longrightarrow} A\ot A \stackrel{\delta_0}{\longrightarrow}
    A\rightarrow 0 
\end{equation}
where 
\begin{equation}\label{eqn: bar-differential}
   \delta_n (a_0\ot \cdots\ot a_{n+1}) = \sum_{i=0}^n (-1)^i a_0\ot \cdots
      \ot a_i a_{i+1}\ot \cdots\ot a_{n+1}
\end{equation}
for all $n\geq 0$ and $a_0,\ldots,a_{n+1}$ in $A$. If
$M= A$, we abbreviate,
$
   \HH^n(A) := \HH^n(A,A).
$

A {\bf deformation of $A$  over $k[t]$} is an associative
$k[t]$-algebra with underlying vector space $A[t]$ and multiplication
determined by
\begin{equation}\label{star-formula}
    a_1 * a_2 = a_1a_2 + \mu_1(a_1\ot a_2) t + \mu_2(a_1\ot a_2) t^2 + \cdots
\end{equation}
where $a_1a_2$ is the product of $a_1$ and $a_2$ in $A$
and each $\mu_k: A\ot A \rightarrow A$ 
is a $k$-linear map.
(Only finitely many terms are nonzero 
for each pair $a_1, a_2$ in the above expansion.)

We record some needed properties of $\mu_1$ and $\mu_2$. 
Note that
$
  \Hom_{k}(A\ot  A, A)\cong \Hom_{A^e}(A\ot A \ot A\ot  A, A)
$
since the $A^e$-module $A\ot A \ot A\ot  A$ is (tensor) induced from the
$k$-module $A\ot A$,
and we identify $\mu_1$ with a  2-cochain
on the  bar resolution~(\ref{relative-bar}).
Associativity of $*$ implies 
that $\mu_1$ is a Hochschild 
2-cocycle,  i.e., that
\begin{equation}\label{eqn:Hoch2cocycle}
  a_1\mu_1(a_2\ot a_3) + \mu_1(a_1\ot a_2a_3) =
     \mu_1(a_1a_2\ot a_3) + \mu_1(a_1\ot a_2)a_3
\end{equation}
for all $a_1,a_2,a_3\in A$, 
or, equivalently, that $\delta_3^*(\mu_1)=0$:
One need only expand each side of the equation
$
   a_1 * (a_2 * a_3) = (a_1*a_2)*a_3
$
and compare coefficients of $t$.
Comparing coefficients of $t^2$ instead yields
$$
  \delta_3^*(\mu_2)(a_1\ot a_2\ot a_3) = \mu_1(a_1\ot\mu_1(a_2\ot a_3))
                    -\mu_1(\mu_1(a_1\ot a_2)\ot a_3)
$$
for all $a_1, a_2, a_3\in A$. Thus we consider $\mu_2$ to be a 
cochain on the bar resolution whose coboundary is given as above.
Generally, for all $i\geq 1$, 
\begin{equation}\label{delta3}
  \delta_3^*(\mu_i)(a_1\ot a_2\ot a_3) = 
   \sum_{j=1}^{i-1} \mu_j (\mu_{i-j}(a_1\ot a_2)\ot a_3)
     -\mu_j(a_1\ot \mu_{i-j} (a_2\ot a_3)).
\end{equation}
We call the right side of the last equation
the {\bf $(i-1)$-th obstruction}.

Now we recall graded deformations. 
Assume that the algebra $A$ is $\NN$-graded. 
Let $t$ be an indeterminate and extend the grading
to $A[t]$ by assigning $\deg t = 1$.
A {\bf graded deformation of $A$ over $k[t]$}
is a deformation $A_t$ of $A$ which is also graded;
each map
$\mu_j:A\otimes A \rightarrow A$ is necessarily homogeneous of
degree $-j$ in this case.
An {\bf $i$-th level graded deformation of $A$} 
is a graded
associative $k[t]/(t^{i+1})$-algebra $A_i$ 
whose underlying vector space is $A[t]/(t^{i+1})$ and whose
multiplication is determined by 
$$
  a_1*a_2 = a_1a_2 + \mu_1(a_1\ot a_2) t + \mu_2(a_1\ot a_2)t^2 + \cdots
              + \mu_i(a_1\ot a_2)t^i
$$
for some maps
$\mu_j:A\ot A\rightarrow A$
extended to be linear over $k[t]/(t^{i+1})$. 
We call  $\mu_j$ the $j$-th {\bf multiplication map} of the deformation
$A_i$; note that it is homogeneous of degree $-j$ as 
$A_i$ is graded.
We say that an $(i+1)$-st level graded deformation $A_{i+1}$
of $A$ 
{\bf extends} (or {\em lifts})
an $i$-th level graded deformation $A_i$ of $A$
if the $j$-th multiplication maps agree for all $j\leq i$.

For any $\NN$-graded $A$-bimodule $M$, 
the bar resolution~(\ref{relative-bar})  
induces a grading on Hochschild cohomology
$\HH^n(A,M)$ 
in the following way: Let
$$\Hom^i_{k}(A^{\ot n},M)$$ be the space of all homogeneous $k$-linear maps
from $A^{\ot n}$ to $M$ of degree $i$, i.e., 
mapping the degree $j$ component of $A^{\ot n}$ (for any $j$) 
to the degree $j+i$ component of $M$,
where $\deg(a_1\ot\cdots\ot a_n)$ in
$A^{\ot n}$ is $\deg a_1+\ldots+ \deg a_n$
for $a_1,\ldots, a_n$ homogeneous in $A$.
The differential on the bar complex~(\ref{relative-bar})
has degree zero, so the spaces of coboundaries and cocycles
inherit this grading.
We denote the resulting $i$-th graded component of $\HH^n(A,M)$ by
 $\HH^{n,i}(A,M)$, so that 
$$\HH^{n}(A,M) = \bigoplus_{i} \HH^{n,i}(A,M)\ .$$

The following proposition is a consequence of~\cite[Proposition~1.5]{BG}. 

\begin{prop}\label{lemma:obstruction}
First level graded deformations of the algebra $A$
define elements of the Hochschild
cohomology space $\HH^{2,-1}(A)$. 
Two such deformations are (graded) isomorphic if, 
and only if, the corresponding
cocycles are cohomologous. 
All obstructions to lifting an $(i-1)$-th level
graded deformation to the next level
lie in $\HH^{3,-i}(A)$; 
an $(i-1)$-th level deformation lifts to the 
$i$-th level if and only if its
$(i-1)$-th obstruction cocycle
is zero in cohomology.
\end{prop}

We are most interested in deformations of skew group algebras formed
from Koszul algebras with group actions, and we
recall Koszul algebras (over the field $k$) next. 
Let $V$ be a finite dimensional vector space.
Let
$$
  T_k(V) = k\oplus V\oplus (V\ot V) \oplus 
(V\ot V\ot V) \oplus\ldots ,
$$
the tensor algebra of $V$ 
over $k$ with $i$-th graded piece $T_k^i(V) := V^{\ot  i}$
and $T_k^0(V)=k$.

Let $R$ be a subspace of $T_k^2(V)$, and let 
$S=T_k(V)/\langle R \rangle$ be the corresponding homogeneous quadratic
algebra.
Let $K^0(S) = k$, $K^1(S) = V$, and for each $n\geq 2$,
define
$$
   K^n(S) = \bigcap_{j=0}^{n-2} ( V^{\ot j}\ot  R\ot V^{\ot (n-2-j)}).
$$
Set $$\widetilde{K}^n(S) = S\ot  K^n(S)\ot  S\ .$$
Then $\widetilde{K}^{0}(S)\cong S\otimes S$,
$\widetilde{K}^{1}(S)= S\otimes V\otimes S$,
$\widetilde{K}^{2}(S)= S\otimes R\otimes S$,
and, in general, $\widetilde{K}^n(S)$ embeds
as an $S^e$-submodule 
into $S^{\ot (n+2)}$, the $n$-th component of the
bar resolution~(\ref{relative-bar}).
We apply the differential $\delta_n$, defined in~(\ref{eqn: bar-differential}),
to an element in $\widetilde{K}^n(S)$ and note that
the terms indexed by $i=1,2,\ldots, n-1$ vanish
(due to the factors in the space of relations $R$).
The remaining terms are clearly in $\widetilde{K}^{n-1}(S)$. 
We may thus consider $\widetilde{K}^{\DOT}(S)$ to be a complex with
differential $d$ (restricted from the bar complex, 
$d_n:= \delta_n |_{{\widetilde{K}}^n(S)}$),
called the {\bf Koszul complex}:
\begin{equation}
\label{relativeKoszul}
\ldots\stackrel{d_3}{\longrightarrow}
\widetilde{K}^{2}(S)\stackrel{d_2}{\longrightarrow}
\widetilde{K}^{1}(S) \stackrel{d_1}{\longrightarrow}
\widetilde{K}^{0}(S)\stackrel{d_0}{\longrightarrow} 
 S\rightarrow 0\ .
\end{equation}
By definition, $S$ is a {\bf Koszul algebra} if the related complex
$\bar{K}^{\DOT}(S)$, defined by $\bar{K}^n(S):= S\ot K^n(S)$, is a
resolution of the $S$-module $k$ on which each element of $V$
acts as 0. 
(Note that $\bar{K}^n(S)\cong \widetilde{K}^n(S)\ot_S k$;
differentials are $d_n\ot \id$.)
It is well-known that $S$ is a Koszul algebra if, and only if, 
$\widetilde{K}^{\DOT}(S)$ is a resolution of the $S^e$-module $S$.
(See e.g.~\cite[Proposition A.2]{BG} or~\cite{Kr}.)


\section{Resolutions for skew group algebras}\label{sec:resolution}

In this section, we consider any finite group $G$ and any $k$-algebra
$S$ upon which $G$ acts by automorphisms.  We work in
arbitrary characteristic to develop techniques helpful in the modular setting.
We explain how to construct a resolution of the
skew group algebra $A=S\# G$ from resolutions of $S$ and of $kG$,
generalizing a construction of
Guccione, Guccione, and Valqui~\cite[\S4.1]{GGV}.
It should be compared to the double complexes in 
\cite[\S2.2]{Sanada} and more generally in \cite[Corollary 3.4]{Stefan}
that are used to build  spectral sequences.
We apply the resolution 
in the next section to generalize the result of 
Braverman and Gaitsgory~\cite[Theorem~4.1]{BG} on the
Poincar\'e-Birkhoff-Witt property.
Some of our assumptions in this section 
will seem restrictive, however the
large class of examples to which we generalize their result
in Theorem~\ref{thm:main}
all satisfy the assumptions, as do the modular versions of 
Drinfeld orbifold algebras, graded Hecke algebras, and symplectic
reflection algebras in Section~\ref{DOAs}.

In characteristic zero, 
$\HH^{\DOT}(S\# G)\cong\HH^{\DOT} (S,S\# G)^G$ as a 
consequence of a spectral sequence argument, and the latter
may be obtained from a resolution of $S$.
(In the special case that $S$ is Koszul, one thus has a convenient resolution at 
hand for analyzing the cohomology of $S\# G$ and its deformation theory.)
But in the modular setting, when the characteristic of $k$
divides the order of $G$, the spectral sequence no longer 
merely produces
$G$-invariants  (as the group may have nontrivial
cohomology) and thus this technique for simplifying the cohomology
of the skew group algebra fails.  
The resolution of $S\# G$ constructed in this section  
retains some of the flavor of the Koszul
resolution of $S$, so as to allow similar Koszul techniques to be applied in
the modular setting. 

We first recall the definition of a {skew group algebra}.
The {\bf skew group algebra} $S\# G$ is a semidirect 
product algebra: It is the $k$-vector space $S\otimes kG$
together with multiplication given by 
$(r \otimes g)(s\otimes h)=r ( \,^{g}s)\otimes g h$
for all $r,s$ in $S$ and $g,h$ in $G$,
where $^gs$ is the image of $s$ under the automorphism $g$.
(We are most interested in skew group algebras arising
from the linear action of $G$ 
on a finite dimensional vector space $V$ and 
its induced actions on 
tensor algebras, symmetric algebras, Koszul algebras, and
homogeneous quadratic algebras all generated
by $V$.)

We need the notion of Yetter-Drinfeld modules:
A vector space $V$ is a {\bf Yetter-Drinfeld module} over $G$ if $V$ is
$G$-graded, that is, $V=\oplus_{g\in G}V_g$, and $V$ is a $kG$-module 
with $h(V_g)= V_{hgh^{-1}}$ for all $g,h\in G$.
Any $kG$-module $V$ is trivially a Yetter-Drinfeld module
by letting $V_1=V$ and $V_g=0$ for all nonidentity group elements $g$.
Similarly, any algebra $S$ on which $G$ acts by automorphisms 
is a Yetter-Drinfeld module in this way.
Alternatively, the group algebra $kG$ itself may be considered to be
a Yetter-Drinfeld module by
letting the $g$-component be all scalar multiples of $g$, for each
$g\in G$, and by letting $G$ act on $kG$ by conjugation.
The skew group algebra $A=S\# G$ is a Yetter-Drinfeld module by
combining these two structures:
$$
   A=\bigoplus_{g\in G} A_g
$$
where $A_g = Sg$ and $G$ acts on $A$ by conjugation (an inner action).
A morphism of Yetter-Drinfeld modules is  a $kG$-module homomorphism that preserves 
the $G$-grading. 
We use a  braiding on the category of Yetter-Drinfeld modules:
Given any two Yetter-Drinfeld modules $V, W$ over $G$, their tensor 
product $V\ot W$ is again a Yetter-Drinfeld module with 
$$
    (V\ot W)_g = \bigoplus_{xy=g} (V_x\ot W_y)
$$
and the usual $G$-action on a tensor product:
$
    {}^g(v\ot w)= \,{}^gv\ot\,  {}^gw
$
for all $g\in G$, $v\in V$, $w\in W$. 
There is an isomorphism of Yetter-Drinfeld modules
$$c_{V,W}: V\ot W\rightarrow W\ot V$$
given by $c_{V,W}(v\ot w) = {}^g w\ot v$ for all $v\in V_g$,
$w\in W$, $g\in G$, called the {\bf braiding}.

We will now construct an $A^e$-free resolution
of $A=S\# G$ from a $(kG)^e$-free resolution of $kG$ and an $S^e$-free
resolution of $S$.
First let 
$$
  \cdots\rightarrow C_1 \rightarrow C_0\rightarrow kG\rightarrow 0
$$
be a $(kG)^e$-free resolution of $kG$.
We assume that  each $C_i$ is $G$-graded with grading preserved by
the bimodule structure, that is $g((C_i)_h)_{l} = (C_i)_{ghl}$ for
all $g,h,l$ in $G$. Note this implies that $C_i$ is a 
Yetter-Drinfeld module where $ {}^g c = gcg^{-1}$ for all $g$ in $G$
and $c$ in $C_i$. 
We also assume that the differentials preserve the $G$-grading. 
Assume that as a free $(kG)^e$-module, $C_i = kG\ot C_i'\ot kG$ for
a $G$-graded vector space $C_i'$ whose $G$-grading induces that on
$C_i$ under the usual tensor product of $G$-graded vector spaces. 
For example, 
the bar resolution of $kG$ satisfies all these properties.
Another instance can be found in Example~\ref{example:cyclic} below.

Next let 
$$
   \cdots\rightarrow D_1\rightarrow D_0\rightarrow S\rightarrow 0
$$
be an $S^e$-free resolution of $S$ consisting of left $kG$-modules for which
the differentials are $kG$-module homomorphisms, and the left actions of $S$
and of $kG$ are compatible in the sense that they induce a left action
of $A=S\# G$. 
We  consider each
$D_i$ to be a Yetter-Drinfeld module by setting it all in the 
component of the identity. 
Write the free $S^e$-module $D_i$ as $S\ot D_i'\ot S$, for
a vector space $D_i'$. 
For example, the bar resolution of $S$ satisfies these properties
under the usual action of $G$ on tensor products. 
If  $S$ is a Koszul algebra on which $G$ acts by graded automorphisms, 
the Koszul resolution of $S$ also satisfies these properties. 

We now induce both $C_{\DOT}$ and $D_{\DOT}$ to $A$ 
by tensoring with $A$ in each degree.
We induce $C_{\DOT}$ from $kG$ to $A$ on the left:
Since $A$ is free as a right $kG$-module under multiplication, and $A\ot _{kG}
kG\cong A$, we obtain an exact sequence of $A\ot (kG)^{op}$-modules, 
$$
   \cdots\rightarrow A\ot_{kG} C_1\rightarrow A\ot_{kG}C_0\rightarrow
      A\rightarrow 0 . 
$$
Similarly, we induce $D_{\DOT}$ from $S$ to $A$ on the right:
Since $A$ is free as a left $S$-module, and $S\ot_S A\cong A$,
$$
  \cdots\rightarrow D_1\ot_S A\rightarrow D_0\ot_S A\rightarrow
   A\rightarrow 0
$$
is an exact sequence of $S\ot A^{op}$-modules.

We extend the actions on the modules in each of these two
sequences so that they become sequences of $A^e$-modules.
This will allow us to take their tensor product over $A$.
We extend the right $kG$-module structure on $A\ot_{kG} C_{\DOT}$ 
to a right $A$-module structure by using the braiding to 
define a right action of $S$:
For all $a\in A$, $g\in G$, $x\in (C_i)_g$, $s\in S$, we set
$$
    (a\ot x) \cdot s := a ({}^gs)\ot x .
$$
We combine this right action of $S$ with the right action of $kG$;
under our assumptions, this results in a right action of $A$ on
$A\ot_{kG} C_i$. To see that it is well-defined, note that  
if $h\in G$, then 
$(ah\ot x)\cdot s = ah ({}^gs) \ot x = a ({}^{hg}s)\ot hx = (a\ot hx)\cdot s$.
Thus $A\ot_{kG} C_i$ is an $A$-bimodule
and the action commutes  with the differentials 
by the assumption that the differentials on $C_{\DOT}$ 
preserve the $G$-grading.

We extend the left $S$-module structure on $D_i\ot_S A$ 
to a left $A$-module structure by defining a left action of
$kG$:
$$
    g\cdot (y\ot a) := {}^g y\ot ga
$$
for all $g\in G$, $y\in D_i$, $a\in A$. 
It is well-defined since $gs = (\,{}^gs)g$ 
for all $s\in S$, and indeed gives a left action
of $A$ on $D_i\ot _S A$.
Again this action commutes with the differentials,
by our assumption that the differentials on $D_{\DOT}$ are
$kG$-module homomorphisms. 

We use the $A$-bimodule structures on 
$A\ot_{kG}C_{\DOT}$ and on $D_{\DOT}\ot_S A$ defined above
and consider each as a complex of $A^e$-modules.
(Note that we have {\em not} assumed  they consist of projective 
$A^e$-modules.) 
We take their tensor product over $A$, setting 
$
   X_{\DOT,\DOT} := (A\ot_{kG}C_{\DOT})\ot_A (D_{\DOT}\ot_S A),
$
that is, for all $i,j\geq 0$, 
\begin{equation}\label{xij}
 X_{i,j}:= (A\ot_{kG}C_i)\ot_A (D_j\ot _{S} A),
\end{equation}
with horizontal and vertical differentials
$$
   d^h_{i,j}: X_{i,j}\rightarrow X_{i-1,j} \ \ \ \mbox{ and }
   \ \ \ d^v_{i,j}: X_{i,j}\rightarrow X_{i,j-1}
$$
given by 
$  d^h_{i,j}:= d_i\ot \id$ and 
$d^v_{i,j}:= (-1)^i\, \id\ot d_j$.
Let $X_{\DOT}$ be the total complex of $X_{\DOT, \DOT}$, i.e., the complex
\begin{equation}\label{resolution-X}
  \cdots\rightarrow X_2\rightarrow X_1\rightarrow X_0\rightarrow A\rightarrow 0
\end{equation}
 with $X_n = \oplus_{i+j=n} X_{i,j}$. 

\begin{thm}\label{thm:resolution}
Let $S$ be a $k$-algebra with action of a finite group $G$ by automorphisms
and set $A=S\# G$.
Let $X_{\DOT}$ be the complex defined in~(\ref{resolution-X}) from factors
$C_i= kG\ot C_i'\ot kG$ and $D_i = S\ot D_i'\ot S$ as above. 
Then $X_{\DOT}$  is a free
resolution of the $A^e$-module $A$, and for each $i,j$,
the $A^e$-module $X_{i,j}$ is isomorphic to $A\ot C_i'\ot D_j'\ot A$.
\end{thm}

In the case that
$C_{\DOT}$ is the (normalized) bar resolution of $kG$ and $D_{\DOT}$ is the
Koszul resolution of $S(V)$, our resolution $X_{\DOT}$ is precisely
that in~\cite[\S4]{GGV}.

\begin{proof}
We first check that for each $i,j$, the $A^e$-module $X_{i,j}$ is free.
By construction, 
\begin{eqnarray*}
   X_{i,j} & = &  (A\ot_{kG}kG\ot C_i' \ot kG) \ot_A 
   (S\ot D_j'\ot S\ot_S A) \\
   & \cong & (A\ot C_i'\ot kG)\ot_A (S\ot D_j'\ot A).
\end{eqnarray*}
We claim that this is isomorphic to $A\ot C_i'\ot D_j'\ot A$ as an
$A^e$-module. To see this, first define a map as follows:
\begin{equation}\label{id1}
  (A\ot C_i'\ot kG)\times (S\ot D_j'\ot A) \rightarrow  
       A\ot C'_i\ot D'_j\ot A
\end{equation}
$$\hspace{2cm}  (a\ot x\ot g , \ s\ot y \ot b) 
   \mapsto  a ( {}^{hg}s) \ot x \ot {}^gy\ot g
$$
where $x\in (C_i')_h$. This map is bilinear by its definition, and we check
that it is $A$-balanced: If $r\in S, \ell\in G$, then on the one hand,
\begin{eqnarray*}
  ((a\ot x\ot g)\cdot (r\ell), \ s\ot y\ot b) & = & 
   (a ({}^{hg}r)\ot x\ot g\ell, \ s\ot y\ot b)\\
     &\mapsto & a ({}^{hg}r) ({}^{hg\ell}s) \ot x\ot {}^{g\ell}y\ot g\ell b ,
\end{eqnarray*}
while on the other hand,
\begin{eqnarray*}
  ((a\ot x\ot g, \ (r\ell)\cdot (s\ot y\ot b)) &=& 
   (a\ot x\ot g, \ r({}^{\ell}s)\ot {}^{\ell}y\ot \ell b)\\
   &\mapsto & a ({}^{hg}r) ({}^{hg\ell}s) \ot x\ot {}^{g\ell}y \ot g\ell b).
\end{eqnarray*}
Therefore there is an induced map 
$$(A\ot C_i'\ot kG)\ot_A (S\ot D_j'\ot A)
\rightarrow A\ot C_i'\ot D_j'\ot A.$$ 
It is straightforward to verify that  an inverse map is given by
\begin{equation}\label{id2}
   a\ot x\ot y\ot b \mapsto (a\ot x\ot 1)\ot (1\ot y\ot b).
\end{equation}
Therefore the two spaces are isomorphic as claimed.

We wish to apply the K\"unneth Theorem, and to that end 
we check that each term in the complex $A\ot_{kG} C_{\DOT}$ consists of
free right $A$-modules, and that the image of each differential
in the complex is projective as a right $A$-module.
This may be proved inductively, starting on one end of the complex
$$
   \cdots \stackrel{f_2}{\longrightarrow} A\ot_{kG}C_1
      \stackrel{f_1}{\longrightarrow}
    A\ot_{kG}C_0 \stackrel{f_0}{\longrightarrow} A \rightarrow 0 .
$$
To see directly that each $A\ot_{kG} C_i$ is free as a right $A$-module, write
$$A\ot_{kG}C_i = A\ot_{kG}(kG\ot C_i' \ot kG)\cong A\ot C_i'\ot kG.$$
Choose a $k$-linear finite basis $\{x_m \mid 1\leq m\leq r_i\}$ 
of $C_i'$ for which each $x_m$ is homogeneous
with respect to the $G$-grading, and $r_i = \dim_k C_i'$. 
(A similar idea works if the $C_i'$ are infinite dimensional, however since
$G$ is finite, it is reasonable to assume that the $C_i'$ are finite
dimensional.)
Then a set of free generators 
of $A\ot_{kG} C_i$ as a right $A$-module is 
$$\{g\ot x_m \ot 1\mid g\in G, 1\leq m\leq r_i \}.$$
Indeed, if we fix $g$ in $G$ and $x_m$ as above, with $x_m$ in the $\ell$-component 
($\ell\in G$), then for each $s$ in $S$ and $h$ in $G$, 
$$(g\ot x_m\ot 1)\cdot ( {}^{\ell^{-1} g^{-1}}s)h = sg\ot x_m\ot h,$$
and consequently the full subspace 
$Sg\ot x_m\ot kG$ is generated from this single element.
It also follows that they are independent.  
Since $A$ is right $A$-projective, and $f_0$ is surjective, the map $f_0$ splits
so that $\ker f_0$ is a direct summand of $A\ot_{kG} C_0$ as a right $A$-module. 
Therefore $\ker f_0=\im f_1$ is right $A$-projective.
Repeat the argument with $A\ot_{kG}C_0$ replaced by $A\ot_{kG}C_1$ 
and $A$ replaced by $\im f_1$, and so on, to complete
the check.

The K\"unneth Theorem~\cite[Theorem~3.6.3]{W}
then gives for each $n$ an exact sequence
$$
 0\longrightarrow \bigoplus_{i+j=n} \coh_i(A\ot_{kG}C_{\DOT})\ot_A 
  \coh_j(D_{\DOT}\ot_S A)
\longrightarrow \coh_n((A\ot_{kG}C_{\DOT})\ot_A (D_{\DOT}\ot_S A))$$

\vspace{-.4cm}

$$ \hspace{4.45cm}
   \longrightarrow \bigoplus_{i+j=n-1} 
\Tor^A_1(\coh_i(A\ot_{kG} C_{\DOT}),
  \coh_j (D_{\DOT}\ot_S A))\rightarrow 0.
$$
Now $A\ot_{kG}C_{\DOT}$ and $D_{\DOT}\ot _S A$ are exact other than in
degree 0, where their homologies are each $A$.
Thus 
$\coh_j(A\ot_{kG} C_{\DOT})=
  \coh_i (D_{\DOT}\ot_S A)=0$ unless $i=j=0$.
The Tor term for $i=j=0$ is also zero as
$\Tor^A_1(A,A)=0$ (since $A$ is flat over $A$).
Thus
$$\coh_n ((A\ot_{kG}C_{\DOT})\ot_A (D_{\DOT}\ot_S A))=0
\text{ for all } n>0$$
and
$$
  \coh_0((A\ot_{kG} C_{\DOT})\ot _A (D_{\DOT}\ot_S A))\cong 
  \coh_0(A\ot_{kG}C_{\DOT})\ot_A
   \coh_0 (D_{\DOT}\ot_S A)\cong A\ot_A A\cong A.
$$
Thus $X_{\DOT}$ is an $A^e$-free resolution of $A$.
\end{proof}

The resolution $X_{\DOT}$, being more general than 
the one given in~\cite[\S4]{GGV}, has an advantage: 
One may use any convenient 
resolution of the group algebra $kG$ in the construction. 
The resolution in~\cite{GGV} uses the (normalized) bar resolution of $kG$,
resulting in a potentially larger complex $X_{\DOT}$.
In the example below, we show  that $X_{\DOT}$ may be quite tractable
when a smaller resolution of $kG$ is chosen. 

\begin{example}\label{example:cyclic}
{\em
Let $G$ be a cyclic group of prime order $p$, generated by $g$.
Let $k$ be a field of characteristic $p$.
Let $V=k^2$ with basis $v_1,v_2$.
Let $g$ act as the matrix
$$
   \left(\begin{array}{cc} 1&1\\0&1\end{array}\right)
$$
on the ordered basis $v_1,v_2$.
Let $S=S_k(V)$, the symmetric algebra, and let
$$
  D_{\DOT} : \hspace{1cm} 0\rightarrow S\ot \Wedge^2V \ot S \rightarrow
  S\ot \Wedge^1 V \ot S \rightarrow S\ot S\rightarrow S\rightarrow 0
$$
be the Koszul resolution of $S$, where we identify $\Wedge^1 V$ with
$V$ and $\Wedge ^2 V$ with $R=\{ v\ot w - w\ot v \mid v,w\in V\}$. 
The differentials commute with the $G$-action.
Let 
$$
  C_{\DOT}: \hspace{.4cm} \cdots \stackrel{v\cdot}{\longrightarrow} kG\ot kG
 \stackrel{u\cdot}{\longrightarrow} kG\ot kG
 \stackrel{v\cdot}{\longrightarrow} kG\ot kG
 \stackrel{u\cdot}{\longrightarrow} kG\ot kG
   \stackrel{m}{\longrightarrow} kG\rightarrow 0
$$
where $u=g\ot 1 -1\ot g$, $v=g^{p-1}\ot 1 + g^{p-2}\ot g + \cdots + 
1\ot g^{p-1}$, and $m$ is multiplication.
Then $C_{\DOT}$ is a $(kG)^e$-free resolution of $kG$; exactness may be
verified by constructing an explicit contracting homotopy.
We consider $kG\ot kG$ in even degrees to be a Yetter-Drinfeld module in
the usual way: $(kG\ot kG)_{g^i} = \Span_k\{x\ot y\mid xy=g^i \}$. But in
odd degrees let $(kG\ot kG)_{g^i} = \Span_k\{x\ot y\mid xy=g^{i-1}\}$.
This will ensure that the differentials preserve the $G$-grading. 

By Theorem~\ref{thm:resolution},
$
   X_{\DOT,\DOT} = (A\ot _{kG}C_{\DOT}) \ot_A (D_{\DOT}\ot_S A)
$
yields a free $A^e$-resolution $X_{\DOT}$  (the total complex) of $A$. 
By our earlier analysis, for all $i\geq 0$ and $0\leq j\leq 2$,  
$$
   X_{i,j}    \cong A\ot \Wedge^jV\ot A
$$
and the differentials are 
$
   d=d^h_{i,j}+d^v_{i,j} = d_i \ot \id + (-1)^i\id\ot d_j.$
}\end{example}

\quad

For our applications to Koszul algebras, we will need chain maps between the resolution
$X_{\DOT}$ and the bar resolution of $A$ with the properties stated
in the next lemma. 

\begin{lemma}\label{lem:maps}
Let $S$ be a finitely generated graded Koszul algebra over $k$
on which a finite group $G$ acts by graded automorphisms.
Let $C_{\DOT}$ be the bar resolution of $kG$, let $D_{\DOT}$
be the Koszul resolution of $S$, let $A=S\# G$, and 
let $X_{\DOT}$ be as in~(\ref{resolution-X}). 
Then there exist chain maps $\phi_{\DOT}: X_{\DOT}\rightarrow A^{\ot (\DOT +2)}$
and $\psi_{\DOT}: A^{\ot (\DOT + 2)}\rightarrow X_{\DOT}$  of degree 0
for which $\psi_n \phi_n$ is the identity map on the
subspace $X_{0,n}$ of $X_n$ for each $n\geq 0$. 
\end{lemma}

A chain map was given explicitly in case
$S=S_k(V)$ by Guccione, Guccione, and Valqui~\cite[\S4.2]{GGV}.

\begin{proof}
Both $X_{\DOT}$ and $A^{\ot(\DOT +2)}$ are free resolutions of
$A$ as an $A^e$-module whose differentials are maps of degree $0$.
We first argue inductively that 
there exists a chain map 
$\phi_n:X_n\rightarrow A^{\ot(n+2)}$
of degree 0 for which $\phi_n |_{X_{0,n}}$ is induced by
the standard embedding of the Koszul complex into the bar complex. 

Suppose $S$ is generated by a finite dimensional $k$-vector space $V$
with quadratic relations $R$: $S=T_k(V)/\langle R \rangle$
(see Section~\ref{sec:relative}).
Define $\phi_0 = \id\ot\id = \psi_0$, the identity map
from $A\ot A$ to itself.
Consider $X_{0,\DOT}$ as a subcomplex (not necessarily acyclic)
of $X_{\DOT}$. 
An inductive argument shows that 
we may  define $\phi_{\DOT}$
so that when restricted to $X_{0,\DOT}$ it corresponds to the standard
embedding of the Koszul complex into the bar complex:
For $n=1$, this is the embedding of $A\ot V\ot A$ into $A\ot A\ot A$,
and one checks that $\phi_1$ on 
$X_1 = X_{0,1}\oplus X_{1,0} \cong
(A\ot V\ot A) \oplus (A\ot kG\ot A)$ may be defined by
$\phi_1(1\ot v\ot 1) = 1\ot v\ot 1$ and $\phi_1(1\ot g\ot 1)=1\ot g\ot 1$
for all $v\in V$, $g\in G$. 
For $n\geq 2$,
\begin{equation}\label{X0n}
   X_{0,n}\cong A\ot \left( \bigcap _{i=0}^{n-2} V^{\ot i}
   \ot R\ot V^{\ot (n-i-2)}\right) \ot A\, ,
\end{equation}
a free $A^e$-submodule of  $A^{\ot (n+2)}$ by its definition.
For each $i,j$ with $i+j=n$, 
choose a basis of the vector space $C_i'\ot D_j'$, 
whose elements are necessarily of degree $j$.
By construction, after applying $\phi_{n-1}d_n$ to these basis elements,
we obtain elements of degree $j$ in the kernel of $\delta_{n-1}$,
which is the image of $\delta_n$.
Choose corresponding elements in the inverse image of $\Ima(\delta_n)$
to define $\phi_n$. 
If we start with an element in $X_{0,n}$, we may choose its canonical image in
$A^{\ot (n+2)}$ (see (\ref{X0n})). 
Elements of $X_{i,j}$ ($i>0$) have
different degree, so their images under $\phi_n$ may be chosen
independently of those of $X_{0,n}$.

Now we show inductively 
that each $\psi_n$ may be chosen to be a degree 0 map for which 
$\psi_n\phi_n$ is the identity map on $X_{0,n}$.
In degree 0, this is true as $\phi_0, \psi_0$ are identity maps.  
In degree 1, $X_{0,1}\cong A\ot V\ot A$ and
$X_{1,0}\cong A\ot kG\ot A$. 
Note that $V\oplus kG$ is a direct summand of $A$ as a vector space. 
We may therefore define
$\psi_1(1\ot v\ot 1)= 1\ot v\ot 1$ in $X_{0,1}$ for all $v\in V$ and 
$\psi_1(1\ot g\ot 1) = 1\ot g\ot 1$ in $X_{1,0}$ for all $g\in G$.
We define $\psi_1$ on elements of the form $1\ot z\ot 1$, for $z$ ranging
over a basis of a 
chosen complement of $V\oplus kG$ as a vector subspace of $A$, 
arbitrarily subject to the condition that $d_1\psi_1(1\ot z\ot 1)
=\psi_0\delta_1(1\ot z\ot 1)$. 
Since $\psi_0, d_1, \delta_1$ all have degree 0 as maps, one may
also choose $\psi_1$ to have degree 0. Note that 
$\psi_1\phi_1$ is the identity map on $X_{0,1}$.
(In fact, it is the identity map on all of $X_1$.) 
Now let $n\geq 2$ and assume that $\psi_{n-2}$ and $\psi_{n-1}$ have
been defined to be degree 0 maps for which $d_{n-1}\psi_{n-1}=
\psi_{n-2}\delta_{n-1}$ and $\psi_{j}\phi_{j}$ is the identity map
on $X_{0,j}$ for $j= n-2, n-1$. To define $\psi_{n}$,
first note that $A^{\ot (n+2)}$ contains as an $A^e$-submodule 
the space $X_{0,n}$ (see~(\ref{X0n})) and the image of each $X_{i,j}$
under $\phi_n$ ($n=i+j$, $i\geq 1$).
By construction, their images intersect in 0 (being generated by
elements of different degrees), the image of $X_{0,n}$
under $\phi_n$ is free, and moreover $\phi_n$ is injective on
restriction to $X_{0,n}$.  
Choose a set of free generators of $\phi_n(X_{0,n})$, 
and choose a set of free generators
of its complement in $A^{\ot (n+2)}$. 
For each chosen generator $x$ of $X_{0,n}$, we 
define $\psi_n (\phi_n (x))$ to be $x$. 
Since  $d_{n}(x)$ is in $X_{0,n-1}$ by definition, 
we have by induction $\psi_{n-1}\phi_{n-1}d_{n}(x)= d_{n}(x)$.
As  $\delta_n\phi_n(x)=\phi_{n-1} d_{n}(x)$, we now have 
$d_n\psi_n\phi_n(x) = \psi_{n-1}\delta_n \phi_n(x)$. That is,
on these elements, 
$\psi_n$ extends the chain map from degree $n-1$ to degree $n$. 
On the remaining
free generators of $A^{\ot (n+2)}$, define $\psi_n$ arbitrarily
subject to the requirement that it be a chain map of degree 0.
\end{proof}

\section{Deformations of quadratic algebras}
\label{sec:main}

Let $B$ be an arbitrary $k$-algebra. 
Let $U$ be a $B$-bimodule that is free as a left
$B$-module and as a right $B$-module, and 
set
$$
   T :=T_B (U) = B\oplus U\oplus (U\ot_B U) \oplus 
(U\ot_B U\ot_B U) \oplus\cdots\, ,
$$
the tensor algebra of $U$ over $B$
with $i$-th graded component $T^i:= U^{\ot_B i}$ and $T^0=B$.
Let $F^i(T)$ be the $i$-th filtered component:
$F^i(T) = T^0\oplus T^1\oplus\cdots\oplus T^i$.
We call a $B$-subbimodule $P$ of $F^2(T)$ a
set of {\bf filtered quadratic relations over $B$}
and we call the quotient $T_B(U)/\langle P \rangle$
a {\bf filtered quadratic algebra over $B$
generated by $U$}.
By Proposition~\ref{uniquePBWrelations},
if the relations are of PBW type, then they are unique.
Set $R=\pi(P)$ where (recall)
$\pi$ is the projection $F^2(T)\rightarrow U\ot_B U$,
so that $T/\langle R \rangle$ is the homogeneous quadratic
algebra determined by $P$.
Note that $R$ is a $B$-subbimodule of $U\ot_B U$. 

We give below some conditions sufficient to guarantee that $P$
and the quadratic algebra $T_B(U)/\langle P \rangle$
it defines are of PBW type. First we present two lemmas. 
It is not difficult to see that any quadratic algebra over $B$
of PBW type
must be defined by a $B$-subbimodule $P\subset T$ 
devoid of elements of filtered degree one.
We record this observation and more in the next lemma.
We choose labels consistent with those in~\cite{BG}
for ease of comparison. 
The proof (see~(\ref{rawIJ})) 
of Proposition~\ref{uniquePBWrelations} implies:
\begin{lemma}\label{lemma:IJ}
Suppose $P\subset T$ is a set of filtered
quadratic relations over $B$ defining
a filtered quadratic algebra $T/\langle P \rangle$ of PBW type
(with respect to $P$). Then
\begin{itemize}
\item[(I)] $\ P\cap F^1(T) = \{0\}$, and
\item[(J)] $\ (F^1(T) P F^1(T) ) \cap F^2(T) = P$.
\end{itemize}
\end{lemma}

If Condition (I) of Lemma~\ref{lemma:IJ}  holds, 
then each (nonhomogeneous) generating relation defining
the quadratic algebra $T/\langle P \rangle$
may be expressed as a unique element 
of homogeneous degree 2 plus linear and constant terms. 
We record these terms with
functions $\alpha$ and $\beta$:
Condition (I) implies existence of $k$-linear maps
$\alpha: R \rightarrow U$ and
$\beta: R\rightarrow B$ for which 
$$
  P = \{ x - \alpha(x) -\beta(x) \mid x\in R\}.
$$
Since $P$ is a $B$-subbimodule of $T$, so is $R$,
and it is not hard to see that
the maps $\alpha$ and $\beta$ are $B$-bimodule homomorphisms.
We may now use the maps $\alpha$ and $\beta$ to 
explore the PBW property using cohomology
instead of examining overlap polynomials and 
ambiguities (see~\cite{AlgorithmicMethods}, for example)
explicitly.
Note that since $U$ is free (and thus flat) as a left $B$-module and
as a right $B$-module, the spaces $R\ot_B U$ and $U\ot_B R$ may be
identified with subspaces of $U^{\ot _B 3}$.

\begin{lemma}\label{lemma:3condns}
Suppose $P\subset T$ is a set of filtered quadratic relations over $B$
defining
a filtered quadratic algebra $T/\langle P\rangle$ of PBW type
(with respect to $P$). Then 
\begin{itemize}
\item[(i)] $\ \Ima (\alpha\ot_B\id - \id\ot_B\alpha) \subset R$, 
\item[(ii)] $\ \alpha\circ (\alpha\ot_B \id - \id\ot_B\alpha) = 
- (\beta\ot_B \id 
                 -\id\ot_B\beta)$,
\item[(iii)] $\ \beta\circ (\alpha\ot_B\id - \id\ot_B\alpha)= 0$,
\end{itemize}
where the maps $\alpha\ot_B\id - \id\ot_B\alpha$ and
$\beta\ot_B\id - \id\ot_B\beta$ are defined
on the subspace $(R\ot_B U)\cap (U\ot_B R)$ of $T$. 

\end{lemma}

\begin{proof}
By Lemma~\ref{lemma:IJ}, Conditions (I) and (J) hold. We 
show that (I) and (J) imply (i), (ii), and (iii).  
Let $x\in (R\ot_B U)\cap (U\ot_B R)$. By definition of $\alpha$ and $\beta$, 
\begin{eqnarray*}
  x - (\alpha\ot_B \id + \beta\ot_B\id) (x)  & \in & PT^1 \ \subset \ F^1(T)PF^1(T),\\
 x - (\id\ot_B\alpha + \id\ot_B\beta) (x) & \in & T^1 P \ \subset \ F^1(T)PF^1(T) .
\end{eqnarray*}
We subtract these two expressions and check degrees
to see that Condition~(J) implies
$$
   (\alpha\ot_B\id - \id\ot_B\alpha + \beta\ot_B \id - \id\ot_B \beta)(x) \ \ \in \ \ 
    (F^1(T)PF^1(T))\cap F^2(T) \ = \  P.
$$
Again considering the degrees of the above elements, we must have
\begin{eqnarray*}
 (\alpha\ot_B\id - \id\ot_B\alpha) (x)  & \in & R,\\
 \alpha((\alpha\ot_B\id - \id\ot_B \alpha)(x)) & = & - (\beta\ot_B\id -\id\ot_B \beta)(x),\\
    \beta((\alpha\ot_B\id - \id\ot_B\alpha)(x)) & = & 0.
\end{eqnarray*}

\end{proof} 

\begin{example}{\em 
We return to Example~\ref{cuteexample} in which
$B=k$ and the $B$-bimodule $U$
is the vector space $V=\Span_k\{x,y\}$
with $P'=\Span_k\{xy-x, yx-y, x^2-x, y^2-y\}$.
As $\beta$ is identically zero, an easy check of Conditions
(i), (ii), and (iii) amounts to checking overlap relations
in $P'$ and verifying that $P'$ is a noncommutative Gr\"obner basis
for the ideal it generates in $T_k(V)$.
}\end{example}

In the remainder of this section, we turn to the case $B=kG$, a finite
group algebra. We show in the next theorem 
that the above conditions, adapted from Braverman and Gaitsgory~\cite{BG}, 
are both necessary and sufficient in the case that 
the homogeneous quadratic algebra determined by $P$ is isomorphic to
$S\# G$ for some Koszul algebra $S$.
(Precisely, we set $U=V\ot kG$ for a finite dimensional $k$-vector
space $V$ and view $U$ as a bimodule over the group algebra
$B=kG$.)
\begin{thm}\label{thm:main}
Let $S$ be a finitely generated
graded Koszul algebra over $k$
on which a finite group $G$ acts by graded
automorphisms.
Suppose a filtered quadratic algebra $A'$
over $kG$ is defined by a
set of filtered quadratic relations
that determine a homogeneous quadratic
algebra isomorphic to $S\# G$.
Then $A'$ is of PBW type 
if and only if Conditions (I), (i), (ii), and (iii)
hold. 
\end{thm}

\begin{proof}
Suppose the Koszul algebra $S$ is generated by the $k$-vector space $V$
with some $k$-vector space of quadratic relations $R'\subset V\ot V$, i.e.,
$S=T_k(V)/\langle R' \rangle$.  
Let $U= V\ot kG$, a $kG$-bimodule with right action
given by multiplication on the rightmost factor $kG$ only 
and left action given by 
$ {}^g (v\ot h) = {}^g v \ot gh$ for all $v\in V$,
$g,h\in G$. 
Set  $R= R'\ot kG$, similarly a $kG$-bimodule
(as $R'$ is a $kG$-module).
Then
$$
  T_{kG}(U)/\langle R\rangle \cong
    (T_k(V)/\langle R'\rangle ) \# G = S\# G
    \cong
    (T_k(V)\# G)/\langle R'\rangle,
$$
as a consequence of the canonical isomorphism between $T_{kG}(U)$
and $T_k(V)\# G$ given by identifying elements of $V$ and of 
$G$ and moving all group elements
far right;
here $\langle R'\rangle$ denotes the ideal of $T_k(V)$ or of $T_k(V)\# G$
generated by $R'$, respectively. 

We may assume that $A'$ is also generated by $U$ over $kG$,
i.e., $A'=T_{kG}(U)/\langle P \rangle$ for a set
of filtered quadratic relations $P$ over $kG$.
Note that $\pi(P)=R$ since
both $\pi(P)$ and $R$ are $kG$-bimodules generating
the same ideal in $T_{kG}(U)$ 
(use Proposition~\ref{uniquePBWrelations}, for example).
The conditions are then 
necessary by Lemmas~\ref{lemma:IJ} and~\ref{lemma:3condns}.
It remains to prove that they are sufficient, so
assume Conditions (I), (i), (ii), and (iii) hold for
$P= \{ x - \alpha(x) -\beta(x) \mid x\in R\}$.

We adapt the proof of Braverman and Gaitsgory~\cite[\S 4]{BG} to our setting
using the resolution $X_{\DOT}$ given by~(\ref{resolution-X})
after choosing $C_{\DOT}$ to be 
the bar resolution~(\ref{relative-bar}) of $kG$
and $D_{\DOT}$ to be the Koszul resolution~(\ref{relativeKoszul}) of $S$.
By Theorem~\ref{thm:resolution}, $X_{\DOT}$
calculates the Hochschild cohomology $\HH^{\DOT}(A)$
cataloging deformations
of $A$.
We first extend the maps $\alpha$ and $\beta$ to cochains on $X_{\DOT}$.
We then use chain maps between $X_{\DOT}$ and the bar resolution
for $A$ to convert $\alpha$ and $\beta$ to Hochschild 2-cochains
which can define multiplication maps for a potential deformation.  
We modify the cochains as necessary to preserve the conditions of the theorem.
Using these conditions, we build a second level graded deformation of
$A$.  We then extend to a graded deformation of $A$ and conclude 
the PBW property
using properties of the resolution $X_{\DOT}$.
We note that Conditions (i), (ii), and (iii) may be interpreted as conditions
on a tensor product over $k$ for the extensions of the
maps $\alpha, \beta$ to $X_{\DOT}$. 

We first extend $\alpha$ and $\beta$ to $X_{\DOT}$.
In degree 2, $X_2$ contains the direct summand
$X_{0,2}\cong A\ot R'\ot A$ (apply Theorem~\ref{thm:resolution}
with $i=0$, $j=2$). 
Note that $R=R'\ot kG\subseteq V\ot V\ot kG\cong U\ot_{kG} U$,
and we thus view the
$kG$-bilinear maps $\alpha, \beta: R\rightarrow A$ as maps on 
$R'\ot kG$. 
Extend them to $A^e$-module maps
from $A\ot R'\ot A\cong A\ot R\ot S$ to $A$
by composing with the multiplication map, 
and, by abuse of notation, denote these extended maps by $\alpha, \beta$
as well. 
Set $\alpha$ and $\beta$ equal to  
0 on the summands $X_{2,0}$ and $X_{1,1}$ of $X_2$
so that  
they further extend to maps $\alpha,\beta: X_2 \rightarrow A$. 

Condition (i) implies that $\alpha$ is 0 on the image of the differential on
$X_{0,3}$.
Since $\alpha$ is a $kG$-bimodule homomorphism by its definition, 
$\alpha$ is $G$-invariant.
We claim that this implies
it is also 0 on the image of the differential on $X_{1,2}$:
Let $a,b\in A$, $g\in G$, and $r\in R'$, and
consider $a\ot g\ot r\ot b$ as
an element of $X_{1,2}\cong A\ot kG\ot R'\ot A$ (apply Theorem~\ref{thm:resolution}
with $i=1$, $j=2$). 
We apply~(\ref{id2}): 
\begin{eqnarray*}
   d(a\ot g\ot r\ot b) & = & d((a\ot g\ot 1)\ot (1\ot r\ot b))\\
    &=& d(a\ot g\ot 1)\ot (1\ot r\ot b) - (a\ot g\ot 1)\ot d(1\ot r\ot b).
\end{eqnarray*}
The second term lies in $X_{1,1}$, but 
$\alpha$ is 0 on $X_{1,1}$ by definition. 
Therefore 
\begin{eqnarray*}
  \alpha(d(a\ot g\ot r\ot  b)) 
    &=& \alpha((ag\ot 1 - a\ot g)\ot (1\ot r\ot b))\\
    &=& \alpha(ag\ot r\ot b - a\ot  {}^gr\ot gb)\\
   &=& ag\alpha(r) b - a \alpha({}^gr) gb,
\end{eqnarray*}
where we use~(\ref{id1}). 
But $g\alpha(r) = \alpha({}^gr) g$, since $\alpha$ is $G$-invariant,  
and thus $\alpha$ is zero on the image of $d$ on $X_{1,2}$. 
It follows that $\alpha$ is a 2-cocycle on $X_{\DOT}$
and defines a Hochschild cohomology class of $\HH^2(A)$.
Thus $\alpha$ yields a first level deformation
$A_1$ of $A$ (i.e., an infinitesimal deformation of $A$) 
with some
first multiplication map $\mu_1: A\ot A\rightarrow A$. 

In fact, we may apply the chain map $\psi_{\DOT}$ of Lemma~\ref{lem:maps}
and choose $\mu_1=\psi_2^* (\alpha)$. 
Note that $\alpha$
is homogeneous of degree $ - 1$ by its definition, and therefore
so is $\mu_1$. 
We claim that $\phi_2^*(\mu_1)=\alpha$ as cochains.
To verify this, first let $x\in X_{0,2}$. 
By Lemma~\ref{lem:maps}, $\psi_2\phi_2(x) = x$, and hence
$$
    \mu_1\phi_2(x) = \alpha\psi_2\phi_2(x) = \alpha(x).
$$
Now  let $x$ be a free generator of
$X_{1,1}$ or of $X_{2,0}$, so that it has degree 1 or 0.
Then $\psi_2\phi_2(x)$ has degree 1 or 0, implying that
its component in $X_{0,2}$ is 0, from which it follows that 
$\mu_1\phi_2(x)= \alpha\psi_2\phi_2 (x) = 0 = \alpha(x)$. Therefore $\phi_2^*(\mu_1)=\alpha$. 

Condition (ii) implies that 
$-d_3^*(\beta)=\alpha\circ (\alpha\ot\id - \id\ot \alpha)$
as cochains on $X_{\DOT}$. 
(Again, since $\beta$ is $G$-invariant,
it will be 0 on the image of the differential on $X_{1,2}$.) 
Let $\mu_2 = \psi_2^*(\beta)$. 
By a similar argument as that above for $\alpha$, we have
$\phi_2^*(\mu_2)=\beta$. 
However, we want $\mu_1,\mu_2$ to 
satisfy the differential condition that $\alpha,\beta$ satisfy, i.e., 
\begin{equation}
\label{differentialcondition}
  -\delta^* (\mu_2) = \mu_1\circ (\mu_1\ot \id-\id \ot\mu_1)
\end{equation}
as cochains on the bar resolution. 
We modify $\mu_2$ as necessary to satisfy this
condition.
Let $$\gamma = \delta^* (\mu_2) + \mu_1\circ (\mu_1\ot\id - \id\ot\mu_1).$$
The cochain $\phi^*(\gamma)$ is zero on  $X_{0,3}$ 
by Condition (ii), since the image of $\phi$ on $X_{0,3}$ is contained
in $A\ot ((R'\ot V) \cap (V\ot R'))\ot A$, and $\phi^*(\mu_1)=\alpha$.
Additionally, $\phi^*(\gamma)$ is 0 on $X_{2,1}$ and on $X_{3,0}$
since it is a map of degree $ - 2$. 
To see that it is also 0 on $X_{1,2}$, note that as an 
$A^e$-module, the image of $X_{1,2}$ under $\phi$ is generated
by elements of degree 2.
Since $\alpha$ contains $kG$ in its kernel and $\mu_1=\psi_2^*(\alpha)$, 
the map $\mu_1\circ (\mu_1\ot\id - \id\ot \mu_1)$ 
must be 0 on the image of $X_{1,2}$ under $\phi$. 
Since $\beta$ is $G$-invariant and $\phi^*(\mu_2)=\beta$, 
we have that $\phi^*\delta^* (\mu_2)=d^*\phi^*(\mu_2)=d^*(\beta)$ 
is 0 on $X_{1,2}$.
Therefore $\phi^*(\gamma)$ is 0 on $X_{1,2}$. 
We have shown that $\phi^*(\gamma)$ is 0 on all of $X_3$, and so 
$\gamma$ defines the zero cohomology
class on the bar complex as well, i.e., it is a coboundary.
Thus there is a 2-cochain $\mu$ of degree $-2$ on the bar complex with 
$\delta^* (\mu) =\gamma$.
If we were to replace $\mu_2$ with $\mu_2 -\mu$, 
it would satisfy the desired differential condition~(\ref{differentialcondition}). 
However, $\phi^*(\mu_2-\mu)$  may not agree with $\beta$ on $X_2$. We subtract
off another term:  Since 
$$
  d^*\phi^*(\mu) = \phi^*\delta^*(\mu)=\phi^*(\gamma ) = 0,
$$ the 2-cochain $\phi^*(\mu )$ is a {\em cocycle} on the complex 
$X_{\DOT}$ and thus
lifts to a {\em cocycle} $\mu'$ of degree $-2$ 
on the bar complex, i.e., $\mu'$ satisfies
$\phi^*(\mu')=\phi^*(\mu)$. 
Then
$\phi^*(\mu_2-\mu+\mu') =\beta$ and 
$$
   \delta^* (\mu_2-\mu +\mu')+\mu_1\circ (\mu_1\ot\id - \id\ot\mu_1)
$$
is zero on the bar resolution
as $\delta^*\mu' = 0$ and $-\delta ^* (\mu_2 - \mu) =
   \mu_1\circ (\mu_1\ot\id - \id\ot\mu_1)$.
We hence replace $\mu_2$ by $\mu_2-\mu+\mu'$. 

We have now constructed maps $\mu_1, \mu_2$
satisfying the differential condition~(\ref{differentialcondition}) required
to obtain a second level graded deformation:
There exists a second level graded deformation $A_2$ of $A$
(with multiplication defined by $\mu_2$) extending $A_1$.

By Condition (iii) and degree considerations,
the obstruction
$$
  \mu_2\circ (\mu_1\ot\id - \id\ot \mu_1) + \mu_1\circ (\mu_2\ot \id
   -\id \ot \mu_2)
$$
 to lifting $A_2$ to a third level deformation
$A_3$
is 0 as a cochain on $X_{\DOT}$ under the cochain map $\phi^*$.
Therefore, as a cochain on the bar resolution, 
this obstruction is a coboundary,
and so represents the zero cohomology class in $\HH^3(A)$. 
Thus there exists a 2-cochain $\mu_3$ of degree $-3$
satisfying the obstruction
equation~(\ref{delta3}) for $i=3$,
and the deformation lifts to the third level by
Proposition~\ref{lemma:obstruction}.

The obstruction for a
third level graded deformation $A_3$ of $A$
to lift to the fourth level lies in $\HH^{3, -4}(A)$
(again by Proposition~\ref{lemma:obstruction}).
We apply $\phi^*$ to this obstruction 
(the right side of equation~(\ref{delta3}) with $i=4$) to obtain
a cochain on $X_3$. But there are no cochains of degree $-4$
on $X_3$ by definition (as it is generated by elements of
degree 3 or less), hence the obstruction is automatically zero.
Therefore the deformation may be continued
to the fourth level. Similar arguments show that it can be continued
to the fifth level, and so on. 

Let $A_t$ be the  (graded) deformation of $A$ that we obtain in this manner.
Then $A_t$ is the $k$-vector space $A [t]$
with multiplication
$$
  a_1*a_2 = a_1a_2 + \mu_1(a_1\ot a_2) t + \mu_2(a_1\ot a_2)t^2 + 
\mu_3(a_1\ot a_2) t^3 + \ldots \ , 
$$
where $a_1 a_2$ is the product in the homogeneous
quadratic algebra $A= S\# G \cong 
T_{kG}(U)/ \langle R \rangle\cong (T_k(V) \# G)/\langle R'\rangle$
and each $\mu_i:A\otimes A \rightarrow A$ is a $k$-linear map
of homogeneous degree $-i$.
(The sum terminates for each pair $a_1, a_2$
by degree considerations.)  Then for any $r$ in $R$,
$\mu_1(r)=\alpha(r)$ and $\mu_2(r)=\beta(r)$ by construction,
and $\mu_i(r)=0$ for $i\geq 3$ by considering degrees.

We now argue that $A'$ is isomorphic as a filtered algebra to the fiber
of the deformation $A_t$ at $t=1$.
Let $A'' = (A_t)|_{t=1}$.
First note that $A''$ is generated by $V$ and $G$, since
it is a filtered algebra (as $A_t$ is graded)
with associated graded algebra $A=S\# G$.
Next note that products of pairs of elements in $G$, or of an element of $V$
paired with an element of $G$, are the same in $A''$ 
as they are in $T_k(V)\# G$:
$g*v=gv$, $v*g=vg$, and $g*h=gh$ for all $g,h$ in $G$.
To verify this observation, one need 
only check that $\mu_1$ vanishes on such pairs,
as $A_t$ is a graded deformation with group elements in degree 0
and vectors in degree 1.
But $\mu_1$ must vanish on low degree pairs 
by our construction of chain maps:
$\mu_1=\psi_2^*(\alpha)$, the chain map $\psi_2$ preserves degree,
and $X_{0,2}$ has free basis as an 
$A^e$-module consisting of elements of degree 2.
Thus the canonical surjective algebra homomorphism 
$$
T_{kG}(U)\cong T_k(V)\# G \rightarrow A''
$$
arises, mapping each $v$ in $V$ and $g$ in $G$ to their copies in $A''$.
The elements of $P$ lie in the kernel (by definition of $A''$),
and thus the map induces a surjective algebra homomorphism:
$$
A'=T_{kG}(U)/\langle P \rangle \rightarrow A''.
$$

We claim this map is an isomorphism of filtered algebras.
First compare the dimensions in each filtered component of these two algebras.
Recall that algebra $A'$  has dimension at most that of $S\# G$ in
each filtered component (as its associated graded algebra
is a quotient of $S\# G$). 
Then since there is a surjective algebra homomorphism
from $A'$ to $A''$ and $\gr(A'')=S\# G$, 
$$
    \dim_k (F^m(S\# G)) \geq \dim _k (F^m(A')) \geq \dim_k(F^m(A''))
    = \dim_k(F^m(S\# G))
$$
for each degree $m$, where $F^m$ denotes the $m$-th filtered component. 
Thus the inequalities are forced to be equalities, and $A'\cong A''$,
a specialization of a deformation of $S\# G$.
Consequently $A'$ is of PBW type. 
\end{proof}

The next result points out a correspondence between
PBW filtered quadratic algebras and fibers of graded deformations
of a particular type: 

\begin{cor}
Let $A=S\# G$ for a finitely generated,
graded Koszul algebra $S$ over $k$ carrying the 
action of a finite group $G$ by graded automorphisms. 
Every graded deformation $A_t$ of $A$
for which the kernel of $\mu_1$ contains $kG\ot V$ and $V\ot kG$ 
 has fiber at $t=1$ isomorphic
(as a filtered algebra) to 
a filtered quadratic algebra over $kG$
of PBW type with induced quadratic algebra
isomorphic to $A$.  Conversely,
every such filtered quadratic algebra
is isomorphic to the fiber at $t=1$ of
a graded deformation of $A$ for which the kernel of $\mu_1$
contains $kG\ot V$ and $V\ot kG$. 
\end{cor}
\begin{proof}
Write the algebra $A=S\# G$ as
$T_{kG}(U)/\langle R\rangle $ where $U=V\ot kG$ for
some finite dimensional $k$-vector space $V$ generating $S$
and set of filtered quadratic relations $R\subset U\ot_{kG}U$.
Suppose $A_t$  
is a graded deformation of $A$
with multiplication map $\mu_1$ vanishing on
all $v\ot g$ and $g\ot v$ for $g$ in $G$ and $v$ in $V$.
(Higher degree maps automatically vanish on such input
as $A_t$ is graded.)
We may reverse engineer the filtered quadratic algebra 
$A'=T_{kG}(U)/\langle P \rangle$ by setting
$P=\{r-\mu_1(r)-\mu_2(r): r\in R\}$.  
The argument at the end
of the proof of Theorem~\ref{thm:main} implies that 
$A'$ is of PBW type and isomorphic to the fiber
of $A_t$ at $t=1$ as a filtered algebra, as this fiber
is a quotient of $T_k(V)\# G$.
Conversely, any PBW filtered quadratic algebra satisfies the
conditions of Theorem~\ref{thm:main}; its proof constructs a graded 
deformation $A_t$, for which the kernel of $\mu_1$ contains $kG\ot V$
and $V\ot kG$,
and whose fiber $A''$ at $t=1$ is isomorphic to $A'$ as a filtered algebra.
\end{proof}

\section{Applications: Drinfeld orbifold algebras, graded Hecke algebras,
and symplectic reflection algebras}\label{DOAs}

We now apply our results from previous sections
to Drinfeld orbifold algebras,
which include graded Hecke algebras, 
rational Cherednik algebras, symplectic reflection
algebras, and Lie orbifold algebras as special cases.
These algebras present as a certain kind of quotient
of a skew group algebra.
Let $G$ be a finite group acting linearly on a finite dimensional
$k$-vector space
$V$, and consider the induced action on $S_k(V)$ and on $T_k(V)$. 
Drinfeld orbifold algebras are deformations
of the skew group algebra $S_k(V)\# G$ (see~\cite{doa}).
Many articles investigate their properties and representation theory, 
in particular, when $k$ is the field of real or complex numbers,
when $G$ acts faithfully, and when $G$ acts symplectically.
We assume that the characteristic of $k$ is not $2$ throughout this
section.
We are especially interested in the case when
the characteristic of $k$ divides $|G|$,
where the theory is much less developed.

Drinfeld orbifold algebras arise as quotients of the form
$$
\cH_{\kappa}=T_k(V)\#G/\langle v_1\otimes v_2-v_2\otimes v_1-\kappa(v_1,v_2): 
v_1,v_2\in V\rangle\ $$
where $\kappa$ is a parameter function on $V\tensor V$ taking
values in the group algebra $kG$, or possibly in $V$, 
or some combination of $kG$ and $V$.
We abbreviate $v\otimes 1$ by $v$ and $1\otimes g$ by $g$
in the skew group algebra $T_k(V)\# G$
(which is isomorphic
to $T_k(V)\otimes kG$ as a $k$-vector space).
We take $\kappa$ to be any alternating map
from $V\otimes V$ to the first filtered component
of $T_k(V)\# G$: 
$$\kappa:V\otimes V \rightarrow 
kG \oplus (V\otimes kG)\, , $$
and write $\kappa (v_1, v_2)$ for $\kappa(v_1\ot v_2)$ for ease of notation.  
The associated graded algebra of $\cH_{\kappa}$ is a quotient of 
$S_k(V)\# G $. 
We say that $\cH_{\kappa}$ is a {\bf Drinfeld orbifold algebra}
if it satisfies the Poincar\'e-Birkhoff-Witt property:
$ \ 
\gr(\cH_{\kappa}) \cong S_k(V)\# G
$
as graded algebras.
We explained in~\cite{doa} how every Drinfeld orbifold
algebra defines a formal deformation of $S_k(V)\#G$ and we also explained 
which deformations arise this way explicitly.

Various authors explore conditions on $\kappa$ guaranteeing
that the quotient $\cH_{\kappa}$ satisfies the PBW property.
Such Drinfeld orbifold algebras are often called:
\begin{itemize}
\item
{\bf Rational Cherednik algebras} when $G$ is a real or complex
reflection group 
acting diagonally on $V=X\oplus X^*$,
for $X$ the natural reflection representation,
and $\kappa$ has image in $kG$ and a particular geometric form,
\item
{\bf Symplectic reflection algebras} when $G$ acts on any symplectic
vector space $V$ and $\kappa$ has image in $kG$,
\item
{\bf Graded affine Hecke algebras} when $G$ is a Weyl group
(or Coxeter group) and $\kappa$ has image in $kG$,
\item
{\bf Drinfeld Hecke algebras} when $G$ is arbitrary
and $\kappa$ has image in $kG$,
\item
{\bf Quantum Drinfeld Hecke algebras} when
$G$ is arbitrary and the nonhomogeneous
relation $v_1\otimes v_2 = v_2\otimes v_1 + \kappa(v_1,v_2)$
is replaced by a quantum version $v_i\otimes v_j = q_{ij} v_j \otimes v_i + 
\kappa(v_i,v_j)$ (for a system of quantum parameters
$\{q_{ij}\}$)
and $\kappa$ has image in $kG$,
\item
{\bf Lie orbifold algebras} when $\kappa$ has image in $V\oplus kG$
(defining a deformation of the universal enveloping
algebra of a Lie algebra with group action).
\end{itemize}
Terminology arises from various settings.  Drinfeld~\cite{Drinfeld} originally
defined these algebras for arbitrary groups and
for $\kappa$ with image in $kG$.
Around the same time, Lusztig (see~\cite{Lusztig88, Lusztig}, for example) 
explored a graded version of the affine Hecke algebra for Coxeter groups.
Ram and Shepler~\cite{RamShepler} showed that Lusztig's graded
affine Hecke algebras are a special case of Drinfeld's construction.
Etingof and Ginzburg~\cite{EtingofGinzburg} rediscovered 
Drinfeld's algebras for symplectic groups in the context of orbifold theory.
The general case (when $\kappa$ maps to 
the filtered degree 1 piece of $T_k(V)\#G$ and $G$ acts with
an arbitrary representation) is explored in~\cite{doa};
see~\cite{HOT} in case the field is the real numbers. 

The original conditions of Braverman and Gaitsgory
were adapted and used for determining which 
$\kappa$ define Drinfeld orbifold algebras,
but arguments relied on the fact that 
the skew group algebra $S_k(V)\# G$ is Koszul
as an algebra over the semisimple ring $kG$.
(This was the approach taken by Etingof and Ginzburg~\cite{EtingofGinzburg}.)
Indeed, we used this theory in~\cite{doa} to establish
PBW conditions in the nonmodular setting. 
However the technique fails in modular characteristic
(i.e., when the characteristic of $k$ divides the order of $G$).

The results in previous sections allow us to overlook
the nonsemisimplicity of 
the group algebra in determining 
which quotients $\cH_{\kappa}$
satisfy the PBW property. We
consider each Drinfeld orbifold algebra
as a nonhomogeneous quadratic algebra whose homogeneous version
is the skew group algebra formed from a Koszul algebra
and a finite group.
We are now able to give a new, shorter proof of
Theorem 3.1 in~\cite{doa} using the methods of Braverman and Gaitsgory but
in arbitrary characteristic. Thus we bypass tedious application
of the Diamond Lemma~\cite{Bergman}. (Details of a long computation 
using the Diamond Lemma over arbitrary fields   
were largely suppressed in~\cite{doa}.)

Decompose the alternating map 
$\kappa:V\ot V\rightarrow (k \oplus V)\ot kG\, $ into
its linear and constant parts by writing
$$\kappa(v_1,v_2)
=\sum_{g\in G}\Big(\kappa_g^C(v_1,v_2)+\kappa_g^L(v_1,v_2)\Big)\ot g$$
for maps $\kappa^C_g:V\otimes V\rightarrow k$ and
$\kappa^L_g: V\otimes V\rightarrow V$, for all $g$ in $G$.
Let $\Alt_3$ denote the alternating group on three symbols.

\begin{prop} \cite[Theorem 3.1]{doa} 
Let $G$ be a finite group acting linearly
on a finite dimensional vector space $V$ over a field $k$
of arbitrary characteristic.  
The quotient algebra
$$
\cH_{\kappa}=T_k(V)\#G/\langle v_1\otimes v_2-v_2\otimes v_1-\kappa(v_1,v_2): 
v_1,v_2\in V\rangle\ $$
satisfies the Poincar\'e-Birkhoff-Witt property if and only if
\begin{itemize}
\item
$\kappa$ is $G$-invariant,
\item
$\displaystyle{\sum_{\sigma\in \text{Alt}_3}
\ \kappa^L_{g}(v_{\sigma(2)}, v_{\sigma(3)}) 
(v_{\sigma(1)}-\, ^gv_{\sigma(1)})=0}$ in $S_k(V)$,
\item
$\displaystyle{\sum_{\sigma\in \text{Alt}_3,h\in G}
\ \kappa^L_{gh^{-1}}\Big(v_{\sigma(1)}+\, ^hv_{\sigma(1)}, 
\kappa^L_h(v_{\sigma(2)}, v_{\sigma(3)})\Big)}$\\
  $\hphantom{x} \hspace{2cm}=
\displaystyle{2\sum_{\sigma\in \text{Alt}_3}
\kappa^C_{g}(v_{\sigma(2)}, v_{\sigma(3)})
(v_{\sigma(1)}-\, ^gv_{\sigma(1)})}$,
\item
$\displaystyle{\sum_{\sigma\in \text{Alt}_3,h\in G}
\ \kappa^C_{gh^{-1}}\Big(v_{\sigma(1)}+\, ^hv_{\sigma(1)}, 
\kappa^L_h(v_{\sigma(2)}, v_{\sigma(3)})\Big)=0}$,
\end{itemize}
for all $v_1, v_2, v_3$ in $V$ and $g$ in $G$.
\end{prop}
\begin{proof}
Let $U=V\otimes kG$.
Consider $T_k(V)\# G\cong T_{kG}(U)$ to be a $kG$-bimodule under
the action  
$g_1(v\tensor g_2)g_3=\, ^{g_1}v\tensor g_1 g_2 g_3$
for $g_i$ in $G$ and $v$ in $T_k(V)$.
As before, we filter $T:=T_{kG}(U)$
by setting  $F^i(T)=T^0+T^1+\ldots+ T^i$
where $T^i=U^{\otimes_{kG}\, i}$ for $i>0$
and $T^0=kG$. 
Extend $\kappa$ to a map
$\kappa: T^2_{kG}(U)\rightarrow kG\oplus U$ 
defined by $\kappa((v_1\otimes g_1)\otimes_{kG} (v_2\otimes g_2))
=\kappa(v_1,\, ^{g_1}v_2)\ot g_1 g_2$ for $v_i$ in $V$ and $g$ in $G$.
(Note that $\kappa$ extends to a unique
$kG$-bimodule map $\kappa: T^2_{kG}(U)\rightarrow kG\oplus U$ if, and
only if, $\kappa$ is $G$-invariant: ${}^g (\kappa({}^{g^{-1}}u,
{}^{g^{-1}}v))= \kappa(u,v)$ for all $u,v$ in $V$ and $g$ in $G$. )

We set $P$ to be the generating nonhomogeneous relations
parametrized by $\kappa$: Let
$P$ be the $kG$-subbimodule of $F^2(T)$
generated by all
$$v_1\otimes_{kG} v_2
-v_2 \otimes_{kG} v_1
- \kappa(v_1 , v_2)$$
for $v_1,v_2$ in $V$.
We then have an isomorphism of filtered algebras,
$\cH_{\kappa} \cong
T_{kG}(U)/\langle P \rangle $.
Thus, $\cH_{\kappa}$ satisfies the PBW property
if and only if $T_{kG}(U)/\langle P \rangle$
exhibits PBW type with respect to $P$ as a filtered
quadratic algebra over $kG$.

We now consider the homogeneous version $A$ determined by $P$.
Set $R=\pi(P)$, the $kG$-subbimodule of $T^2_{kG}(U)$ generated by all
$v_1\otimes_{kG} v_2-v_2 \otimes_{kG} v_1$
for $v_1,v_2$ in $V$.
Set $A:=T_{kG}(U)/\langle R \rangle$
and note that $A\cong 
S_k(V)\# G
\cong
T_{kG}(U)/ \langle \pi(P) \rangle$. 
Then as $S_k(V)$ is Koszul,
the  conditions of Theorem~\ref{thm:main}
apply to give explicit conditions on
$\kappa$ under which $\cH_{\kappa}$ is of
PBW type.  We apply the conditions as in the first proof of
Theorem~3.1 of~\cite{doa} without needing the extra
assumption there that $kG$ is semisimple.
\end{proof}

In the next corollary, we set $\kappa^L_g\equiv 0$ for all $g$ in $G$
to recover a modular version of a 
result that is well-known in characteristic zero
(stated in~\cite{Drinfeld}, then confirmed
in~\cite{EtingofGinzburg} and~\cite{RamShepler}).  The positive
characteristic result was first
shown by Griffeth~\cite{Griffeth} 
by construction of an explicit $\cH_{\kappa}$-module,
as in the classical proof of the PBW theorem for 
universal enveloping algebras of Lie algebras.
(See also Bazlov and Berenstein~\cite{BazlovBerensteinBraidedDoubles}
for a generalization.)
Our approach yields a different proof.
Note that several authors (for example,
Griffeth~\cite{Griffeth} and Balagovic and Chen~\cite{BalagovicChen})
study the representations of
rational Cherednik algebras in the modular setting.
\begin{cor}
Let $G$ be a finite group acting linearly
on a finite dimensional vector space $V$ over a field $k$
of arbitrary characteristic.  
Let $\kappa:V\ot V\rightarrow kG$
be an alternating map.  The quotient algebra
$$
\cH_{\kappa}=T_k(V)\#G/\langle v_1\otimes v_2-v_2\otimes v_1-\kappa(v_1,v_2): 
v_1,v_2\in V\rangle\ $$
satisfies the Poincar\'e-Birkhoff-Witt property if
and only if
$\kappa$ is $G$-invariant and
$$
0=\sum_{\sigma\in\text{Alt}_3}
\kappa_g(v_{\sigma(2)}\ot v_{\sigma(3)})\ 
(\, ^gv_{\sigma(1)}-v_{\sigma(1)})
$$
for all $v_1, v_2$ in $V$ and $g$ in $G$.
\end{cor}


\end{document}